\documentclass[11pt]{amsproc}

\usepackage{a4wide}
\usepackage[pagebackref=true,colorlinks,linkcolor=red,citecolor=blue,urlcolor=blue,hypertexnames=true]{hyperref}
\usepackage{setspace}
\usepackage{amsmath}
\usepackage{amssymb}
\usepackage{amsfonts}
\usepackage{amsthm}
\usepackage{mathtools}
\usepackage{array}
\usepackage{comment}

\theoremstyle{plain}
\newtheorem{thm}{Theorem}
\theoremstyle{definition}

\newtheorem{prop}[thm]{Proposition}
\newtheorem{cor}[thm]{Corollary}
\theoremstyle{remark}

\newcommand{\ti}[1]{\textit{#1}}
\newcommand{\tn}[1]{\textnormal{#1}}

\newcommand{\norm}[1]{\left\Vert#1\right\Vert}

\newcommand{\abs}[1]{\left\lvert#1\right\rvert}

\newcommand{\eps}{\varepsilon}
\newcommand{\bbC}{\mathbb{C}}

\newcommand{\bbN}{\mathbb{N}}

\newcommand{\cR}{\mathcal{R}}

\newcommand{\cU}{\mathrm{U}}

\newcommand{\cSU}{\mathrm{S}\mathrm{U}}

\DeclareMathOperator{\diam}{diam}

\DeclareMathOperator{\tr}{tr}
\DeclareMathOperator{\rk}{rk}

\title[Extreme Amenability]{A new proof of extreme amenability of the unitary group of the hyperfinite II$_1$ factor}

\author{Philip A. Dowerk}
\address{P.A.D., Analysis Section, KU Leuven, 3001 Leuven, Belgium}
\email{philip.dowerk@wis.kuleuven.be}

\author{Andreas Thom}
\address{A.T., Institut f\"ur Geometrie, TU Dresden, 01062 Dresden, Germany }
\email{andreas.thom@tu-dresden.de}

\begin{document}

\onehalfspace

\begin{abstract}
 We provide an alternative proof for the extreme amenability of the unitary group of the hyperfinite II${}_1$-factor von Neumann algebra, endowed with the strong operator topology.
\end{abstract}

\maketitle


\section{Introduction}

Ever since the introduction of {\it rings of operators} in the groundbreaking work of Murray and von Neumann \cite{MvN-36, MvN-43}, the hyperfinite II$_1$-factor $\cR$ has played an important role in (what is nowadays called) the theory of von Neumann algebras. The first construction of $\cR$ was given in terms of finite-dimensional matrix algebras or for example as the group von Neumann algebra of the group $S_{\infty}$ of finitary permutations on the set of natural numbers. 
In seminal work of Connes \cite{Co-76}, $\cR$ was shown to be the unique injective factor of type II$_1$. It also followed, that $\cR$ is isomorphic to the group von Neumann algebra $L\Gamma$ for every countable, amenable group $\Gamma$, whose non-trivial conjugacy classes are all infinite.

The direct relationship between the concept of amenability and hyperfiniteness triggered the question in what sense the unitary group $\cU(\cR) = \{u \in \cR \mid uu^*=u^*u=1\}$ of the hyperfinite II$_1$-factor is amenable itself as a topological group. This point was clarified by Giordano-Pestov \cite{GP-06} who showed that $\cU(\cR)$ is extremely amenable.
The aim of this note is to provide a new and direct proof of extreme amenability of the unitary group of the hyperfinite II$_1$ factor, endowed with the strong operator topology. 
Recall, a topological group $G$ is said to be \textbf{extremely amenable} if every continuous action of $G$ on a compact Hausdorff space $X$ admits a fixed point. A nice account on the history of the subject can be found in Pestov's book \cite{Pes-06}.

A milestone in the study of extreme amenability was set by Gromov and Milman \cite{GM-83} - they proved that the unitary group $\cU(\ell^2(\mathbb N))$ with the strong operator topology is extremely amenable.
 The core in their proof is to show that $\cU(\ell^2(\mathbb N))$ with the strong operator topology is a L\'evy group by treating $\cSU(n)$ as a Riemannian manifold and showing that $\inf_{t}\mathrm{Ric}(t,t)\rightarrow \infty$ as $n\rightarrow \infty$, where $t$ runs over all unit tangent vectors in the tangent space of $\cSU(n)$. Together with the isoperimetric inequality this implied that $(\cSU(n),d_n,\mu_n)$ forms a L\'evy family with the respect to the unnormalized Hilbert-Schmidt metric $d_n$ and Haar measure $\mu_n$ on $\cSU(n)$. This rather deep fact was then used in the proof of extreme amenability of $\cU(\cR)$ by Giordano and Pestov \cite{GP-06}. Note that the corresponding statement fails for the family $(\cU(n),d_n,\mu_n)$, since the Ricci curvature vanishes on tangent vectors corresponding to the center of $\cU(n)$.
 
The main purpose of this note is to give a direct argument for the fact that after normalization of the metric the unitary groups do form a L\'evy family. 

\begin{thm}
The family $(\cU(n),d_n/n,\mu_n)_n$ is a L\'evy family. 
\end{thm}

In particular, we do not rely on the relationship with the isoperimetric inequality and on curvature computations. To the best of our knowledge, this direct argument was unnoticed -- and still implies in a straightforward way extreme amenability of $U(\cR)$ using \cite[Theorem 4.1.3]{Pes-06}.

\begin{cor}[Giordano-Pestov \cite{GP-06}]\label{thm_extr_amen}
	The unitary group $\cU(\cR)$ of the hyperfinite II$_1$ factor, endowed with the strong operator topology, is extremely amenable.
\end{cor}

\section{Metric measure spaces}

Recall that a \textbf{space with metric and measure}, or a \textbf{$mm$-space}\index{$mm$-space}, is a triple $(X,d,\mu)$ consisting of a set $X$, a metric $d$ on $X$ and a probability Borel measure on the metric space $(X,d)$.
The \textbf{concentration function} $\alpha_X:[0,\infty)\rightarrow [0,1/2]$ of an \ti{mm}-space $X$ (introduced by Milman and Schechtman in \cite{Mil-86, MS-86}) is defined as
\begin{align*}
	\alpha_X(\varepsilon)=\left\{
	\begin{aligned}
		& 1/2, &\mbox{ if } \varepsilon =0,\\
		&1-\inf\{\mu(A_{\varepsilon})\mid A\subseteq X\tn{ is Borel },\mu(A)\geq 1/2\},&\mbox{ if } \varepsilon >0.
	\end{aligned}
	\right.
\end{align*}
A family $(X_n,d_n,\mu_n)_{n}$ of $mm$-spaces is a \textbf{L\'evy family} if $\alpha_{X_n}(\eps)\rightarrow_{n\rightarrow\infty} 0$ pointwise for all $\eps>0$.
This is not the original definition of a L\'evy family, but it is equivalent.

A metrizable group $(G,d)$ is called \textbf{L\'evy group} if there is a family of compact subgroups $(G_{n})_n$ of $G$, directed by inclusion, with dense union and such that $(G_n,d_n,\mu_n)_{n}$ forms a L\'evy family, where $d_n$ is the restriction of $d$ to $G_n$ and $\mu_n$ is the normalized Haar measure on $G_{n}$.

Let $d_{1,n}$ denote the normalized trace metric on the space $M_{n\times n}(\mathbb{C})$ of $n\times n$-matrices induced from the normalized trace norm $\norm{\cdot}_{1,n}$, where $n\in\bbN$. That is, with $\tr$ the unnormalized trace on $M_{n\times n}(\bbC)$,
$d_{1,n}(u,v)=\norm{u-v}_{1,n}=\frac{1}{n}\tr(\abs{u-v}),\ u,v\in M_{n\times n}(\mathbb{C}).$ Here, $|a| = (a^*a)^{1/2}$ denotes the absolute value of the matrix as usual.
 Denote by $\rk(x)$ the rank of $x\in M_{n\times n}(\mathbb{C})$ and by $\norm{\cdot}_{\infty,n}:=\sup_{\xi\in\bbC^n, \norm{\xi}_n=1}\norm{\cdot\xi}_n$ the operator norm. Note that $\|xy\|_{1,n} \leq \|x\|_{1,n} \|y\|_{\infty,n}$ and hence
 \begin{equation} \label{lem_traceopineq}
\norm{x}_{1,n}\leq \frac{\rk(x)}{n}\norm{x}_{\infty,n}.
\end{equation}

\begin{prop}
Let $1 \leq k\leq n\in\bbN$ and $u\in \cU(k)$. Then there exists $v\in \cU(k-1)$ such that $d_{1,n}(v \oplus 1_{n-k+1},u \oplus 1_{n-k}) \leq 4/n.$
\label{prop_trace}
\end{prop}

\begin{proof}
The cases $k=1,2$ are trivial since $d(1_n,u \oplus 1_{n-2}) \leq 2/n$ for all $u \in U(2)$. Consider the case $k \geq 3$ and assume without loss of generality that $k=n$. Denote by $\{e_k\}_{k=1,\ldots,n}$ the standard orthonormal basis of $\mathbb{C}^n$. If $ue_n=e_n$, then $u\in\cU(n-1)$ and we can choose $v=u^*\in\cU(n-1)$.
Hence assume that $u^*e_n=\xi\neq e_n$ and consider $X:=\mathrm{span}\langle e_n,\xi\rangle \cong \mathbb{C}^2$. There exists an unitary operator $w \colon X\rightarrow X$ such that $w \xi=e_n$. Define $v:=(1_{X^{\bot}}\oplus w)u^*$ with $X^{\bot}$ the orthogonal complement of $X$. Then $v \in\cU(n-1)$ and $x = 1 -vu = 0_{X^{\perp}} \oplus (1_X - w)$. Hence, $\rk(x) \leq 2$, $\|x\|_{\infty,n} \leq 2$ and the estimate \eqref{lem_traceopineq} imply that
$d_{1,n}(1,uv)=\norm{1-vu}_{1,n}=\norm{x}_{1,n}\leq 4/n.$
\end{proof}


Assume that $H$ is a closed subgroup of a compact group $G$, equipped with a bi-invariant metric $d$. Then the formula $\widetilde{d}(g_1H,g_2H):=\inf_{h_1,h_2\in H}d(g_1h_1,g_2h_2)$ defines a left-invariant metric on the factor space $G/H$, see Lemma 4.5.2 in \cite{Pes-06}. We refer to $\widetilde{d}$ as the \textbf{factor metric}\index{factor metric}. Define the \textbf{diameter}\index{diameter} $\diam(G/H)$ of the factor space $G/H$ to be
\begin{align*}
	\diam(G/H):=\sup_{g_1,g_2\in G}\inf_{h_1,h_2\in H}d(g_1h_1,g_2h_2).
\end{align*}



\section{Proof of the main result}

\begin{prop}
	$\left(\cU(n),d_{1,n},\mu_n\right)_{n\in\bbN}$ forms a L\'evy family, where $d_{1,n}$ denotes the normalized trace metric on $\cU(n)$ and $\mu_n$ is the normalized Haar measure on $\cU(n)$.
	\label{prop_levy}
\end{prop}

\begin{proof}
Our proof is based on \cite[Theorem 2.9]{GP-06}, a result of what is called the martingale technique, see \cite[Theorem 7.8]{MS-86}. Consider the compact Lie group $\cU(n),\ 3\leq n\in\bbN,$ equipped with the bi-invariant trace metric $d_{1,n}$. Embed $\cU(k)$ in $\cU(n)$ via $\cU(k)\ni u\mapsto u \oplus 1_{n-k}\in \cU(n)$, where $k\leq n,\ k\in\bbN$. We calculate the diameter $a_k:=\diam(\cU(k)/\cU(k-1))$ of the factor space $\cU(k)/\cU(k-1)$ with regard to the metric inherited from $U(n)$, where $k=1,\ldots,n$. We use Proposition \ref{prop_trace} to obtain
\begin{align*}
		a_k = \sup_{u\in\cU(k)}\inf_{v\in\cU(k-1)} d_{1,n}(1,uv) \leq\frac{4}{n}.
	\end{align*}
Thus \cite[Theorem 2.9]{GP-06} and the above calculations imply that the concentration function of the \ti{mm}-space $(\cU(n),d_{1,n},\mu_n)$ satisfies
 \begin{align*}
  \alpha_{\cU(n)}(\eps)&\leq 2\exp\left( -\frac{\eps^2}{8\sum_{k=0}^{n-1}{a_k^2}}\right) 
  \leq 2\exp\left( -\frac{n^2\eps^2}{8\sum_{k=0}^{n-1}{16}}\right)
  = 2\exp\left( -\frac{n\eps^2}{128}\right). 
 \end{align*}
	Hence, $\alpha_{\cU(n)}\rightarrow_{n\rightarrow\infty}0$ pointwise on $(0,\infty)$ and thus $(\cU(n),d_{1,n},\mu_n)_n$ is a L\'evy family.
\end{proof}

Actually Proposition \ref{prop_trace} and Proposition \ref{prop_levy} hold analogously for the orthogonal groups $\mathrm{O}(n)$, thus $\left(\mathrm{O}(n),d_{1,n},\mu_n\right)_{n}$ forms a L\'evy family.

\begin{thm}\label{thm_Levy}
The group $\cU(\cR)$ with the strong operator topology is a L\'evy group.
\end{thm}
\begin{proof} Consider the realization of $\cR$ as an infinite tensor product of copies of $M_2 \mathbb C$. The inclusion $\otimes_{i=1}^n M_2 \mathbb C \subset \cR$ yields an inclusion $\cU(2^n) \subset \cU(\cR)$.
The directed family $\{\cU(2^n)\}_{n}$ of compact subgroups of $\cU(\cR)$ is strongly dense in $\cU(\cR)$. Moreover,  the strong topology in $\cU(\cR)$ is induced from the 1-norm which restricts to $d_{1,2^n}$ on each $U(2^n)$. By Proposition \ref{prop_levy} $(\cU(2^n),d_{1,2^n},\mu_{2^n})_n$ forms a L\'evy family.
\end{proof}

%
%

\section*{Acknowledgments}

A.T.\ was supported by ERC StG 277728 and P.A.D. was partially supported by ERC CoG 614195 from the European Research Council under the European Union's Seventh Framework Programme. P.A.D.\ wants to thank Universit\"at Leipzig, the IMPRS Leipzig and the MPI-MIS Leipzig for support and an stimulating research environment. The material in this article is part of the PhD-thesis of the first author.

\bibliographystyle{alpha}

\end{document}